\documentclass[11pt,reqno]{amsart}
\usepackage{amsmath,amssymb,latexsym,soul,cite,mathrsfs}
\usepackage{color,enumitem,graphicx}
\usepackage[colorlinks=true,urlcolor=blue,
citecolor=red,linkcolor=blue,linktocpage,pdfpagelabels,
bookmarksnumbered,bookmarksopen]{hyperref}
\usepackage[english]{babel}
\usepackage{enumitem}
\usepackage[left=2.5cm,right=2.5cm,top=2.9cm,bottom=2.9cm]{geometry}

\usepackage[hyperpageref]{backref}

\makeatletter
\newcommand{\leqnomode}{\tagsleft@true}
\newcommand{\reqnomode}{\tagsleft@false}
\makeatother

\numberwithin{equation}{section}

\newtheorem{theorem}{Theorem}[section]
\newtheorem{lemma}[theorem]{Lemma}

\newtheorem{proposition}[theorem]{Proposition}

\newtheorem{remark}[theorem]{Remark}

\newcommand{\R}{\mathbb R}

\renewcommand{\rightarrow}{\to}

\providecommand{\ln}{\mathop{\rm ln}\nolimits}

\title[Compact embedding theorems for fractional Orlicz--Sobolev spaces]{Compact embedding theorems and a Lions' type Lemma for fractional Orlicz--Sobolev spaces }

\author[E.D. Silva]{Edcarlos D. Silva}
\author[M.L. Carvalho]{M. L. M. Carvalho}
\author[J.C. \ de Albuquerque]{J. C. de Albuquerque}
\author[S. Bahrouni]{Sabri Bahrouni}

\address[E.D. Silva]{Department of Mathematics,
	Federal University of Goi\'{a}s
	\newline\indent
	74001-970, Goi\'{a}s-GO, Brazil}
\email{\href{mailto:eddomingos@hotmail.com}{eddomingos@hotmail.com}}

\address[M.L. Carvalho]{Department of Mathematics,
	Federal University of Goi\'{a}s}
\email{\href{mailto:marcos_leandro_carvalho@ufg.br}{marcos$\_$leandro$\_$carvalho@ufg.br}}

\address[J.C. de~Albuquerque]{Department of Mathematics,
	Federal University of Pernambuco
	\newline\indent
	50670-901, Recife-PE, Brazil}
\email{\href{mailto:josecarlos.melojunior@ufpe.br ; dealbuquerquejc@gmail.com}{josecarlos.melojunior@ufpe.br ; dealbuquerquejc@gmail.com}}

\address[S. Bahrouni]{Mathematics Department, Faculty of Sciences, University of Monastir, 5019 Monastir, Tunisia}
\email{\href{mailto: sabribahrouni@gmail.com}{sabribahrouni@gmail.com}}

\thanks{Corresponding author: Edcarlos D. Silva}

\thanks{The first author was partially supported by CNPq/Universal 2018 with grant 429955/2018-9. Research supported in part by INCTmat/MCT/Brazil, CNPq and CAPES/Brazil}


\begin{document}

\begin{abstract}
In this paper we are concerned with some abstract results regarding to fractional Orlicz-Sobolev spaces. Precisely, we ensure the compactness embedding for the weighted fractional Orlicz-Sobolev space into the Orlicz spaces, provided the weight is unbounded. We also obtain a version of Lions' ``vanishing'' Lemma for fractional Orlicz-Sobolev spaces, by introducing new techniques to overcome the lack of a suitable interpolation law. Finally, as a product of the abstract results, we use a minimization method over the Nehari manifold to prove the existence of ground state solutions for a class of nonlinear Schr\"{o}dinger equations, taking into account unbounded or bounded potentials.
\end{abstract}

\maketitle

%\tableofcontents

\section{Introduction}

Nonlinear Schr\"{o}dinger equations naturally arise as models to describe many physical phenomenons in a mathematical point of view. For this reason, the study of such equations has made great progress and attracted many authors attention in the last decades. An important example is the following class of equations
\begin{equation}\label{equa}
  -\Delta u+V(x)u= f(x,u),\quad x\in\mathbb{R}^{N},
\end{equation}
which is a reducing of the well known Schr\"{o}dinger equation after making a standing wave ansatz. The literature contains many existence, nonexistence and multiplicity results about its solutions and we do not even try to review the huge bibliography. When looking for solutions via Variational Methods, an important feature of this class of problems relies on the ``lack of compactness'' inherit by the unbounded domain $\mathbb{R}^{N}$. In order to overcome such difficulty, some authors regaining some sort of compactness working with radial symmetric functions. The first in this direction is due to Strauss \cite{strauss} where the potential $V(x)$ is constant. In this important work it was introduced the nowadays well-known Strauss's compactness Lemma for radial symmetric functions. Another way to deal with the lack of compactness is to consider a potential $V(x)$ which is ``large'' at infinity. In this direction, we cite the seminal work of Rabinowitz \cite{Rabinowitz} where it was considered a class of potentials coercive and bounded away from zero. This class of potentials was extended by Bartsch and Wang \cite{bw} where it was introduced the following assumptions:

\begin{enumerate}[label=($V_0$),ref=$(V_0)$]
	\item \label{(V_0)}
	It holds that $V(x) \geq V_0$ for any $x \in \mathbb{R}^N$ where $V_0 > 0$;
\end{enumerate}

\begin{enumerate}[label=($V_1$),ref=$(V_1)$]
	\item \label{(V_1)}
	The set $\{x \in \mathbb{R}^N ; V(x)  < M\}$ has finite Lebesque measure for each $M > 0$.
\end{enumerate}
Under the above conditions the authors were able to restore some compactness. The works \cite{Rabinowitz,bw} have motivated many works regarding to nonlinear Schr\"{o}dinger equations involving different operators and potentials that are large at infinity. Our main contribution here is to extend some ``classical results'' to a general class of operators which generalizes many problems previously studied.

This work is motivated by a very recent trend in the fractional framework, which is to consider a new nonlocal and nonlinear operator, the so-called \textit{fractional $\Phi$-Laplacian}. Throughout this work, we shall consider $\Phi:\mathbb{R}\rightarrow\mathbb{R}$ an even function defined by
\[
\Phi(t)=\int_0^ts\varphi(s)\;\mathrm{d} s,
\]
where $\varphi:\mathbb{R}\rightarrow\mathbb{R}$ is a $C^1$-function satisfying the following assumptions:
\begin{itemize}
	\item[($\varphi_1$)] $t\varphi(t)\mapsto 0$, as $t\mapsto 0$ and $t\varphi(t)\mapsto\infty$, as $t\mapsto\infty$;
	\item[($\varphi_2$)] $t\varphi(t)$ is strictly increasing in $(0, \infty)$;
	\item[($\varphi_3$)] there exist $\ell,m\in(1,N)$ such that
	\[
	\ell\leq \frac {t^2\varphi(t)}{\Phi(t)}\leq m<\ell^*, \quad \mbox{for all} \hspace{0,2cm} t>0.
	\]
\end{itemize}
For $s \in (0, 1)$ and $u$ smooth enough, the \textit{fractional $\Phi$-Laplacian operator} is defined as
\begin{equation}\label{emj1}
(-\Delta_{\Phi})^{s}u(x):=P.V.\int \varphi\left( |D_s u|\right)\frac{D_{s}u}{|x-y|^{N+s}}\,\mathrm{d}y, \quad \mbox{where} \hspace{0,3cm} D_s u:=\frac{u(x)-u(y)}{|x-y|^s}
\end{equation}
and $P.V.$ denotes the principal value of the integral. Throughout the paper we write $\int u\,\mathrm{d}x$ instead of $\int_{\R^N} u(x)\,\mathrm{d}x$. The fractional $\Phi$-Laplacian extends a large class of known nonlocal operators. For instance, if $\varphi\equiv1$, then \eqref{emj1} becomes
\[
(-\Delta)^{s}u(x) = P.V.\int \frac{(u(x)-u(y))}{|x-y|^{N+2s}}\;\mathrm{d}y,
\]
which is the well known \textit{fractional Laplace operator}. For works involving this class of operators, we refer the interested reader to \cite{Ambrosio,caffa,guia,Simone Secchi,FS,Lazkin}. Furthermore, when $\varphi(t)=t^{p-2}, p \in (1, N)$ then \eqref{emj1} reduces to
\[
(-\Delta_{p})^{s} u(x) = P.V.\int \frac{|u(x)-u(y)|^{p-2}(u(x)-u(y))}{|x-y|^{N+ps}}\;\mathrm{d}y,
\]
which is the \textit{fractional p-Laplace operator}. In a similar way, if $\varphi(t)=t^{p-2}+t^{q-2}$, $1 < p < q < N$, then we have the fractional $(p,q)$-Laplacian operator given by the following way
\[
(-\Delta_{p} - \Delta_q)^{s} u = P.V.\int \frac{|u(x)-u(y)|^{p-2}(u(x)-u(y))}{|x-y|^{N+ps}}\;\mathrm{d}y + P.V.\int \frac{|u(x)-u(y)|^{q-2}(u(x)-u(y))}{|x-y|^{N+qs}}\;\mathrm{d}y.
\]
Due to the generality of the fractional $\Phi$-Laplacian operator \eqref{emj1} and motivated by the very recent papers \cite{Bonder,ABS, Sabri3,Bahrouni-Salort, Bonder-2,AlvesAmbrosio}, mainly taking into account the work of Bonder and Salort \cite{Bonder}, our goal is to study the following class of fractional Schr\"{o}dinger equations
 \begin{equation}\label{p1}
(-\Delta_{\Phi})^{s}u+V(x)\varphi(u)u=f(x,u), \quad x\in\mathbb{R}^{N},\tag{$P$}
\end{equation}
where $N > 2 s$, $0<s<1$ and the nonlinear term $f$ is of $C^{1}$ class and satisfies suitable assumptions. Due to the presence of the potential $V(x)$, we introduce the following suitable weighted fractional Orlicz-Sobolev space
\[
X:=\left\{u\in W^{s,\Phi}(\mathbb{R}^{N}):\int V(x)\Phi(|u|)\;\mathrm{d}x<+\infty\right\},
\]
endowed with the norm
$$\|u\|=[u]_{s,\Phi}+\|u\|_{V,\Phi},$$
where
$$\|u\|_{V,\Phi}=\inf\left\{\lambda>0:\int_{\mathbb{R}^N}V(x)\Phi\left(\frac{u(x)}{\lambda}\right)\mathrm{d}x\leq 1\right\}.$$
It is important to emphasize that $X$ is a reflexive Banach space, see \cite{Bahrouni-Emb}. Furthermore,  $X \subset W^{s, \Phi}(\mathbb{R}^N)$ is a closed set. In order to study Problem \eqref{p1} variationally, we need to deal with the ``lack of compactness'' due to the fact that the embedding $X\hookrightarrow L_A(\R^N)$ is not compact for all $N$-function $A<\Phi_*$. Thus, we consider both cases when $V(x)$ satisfies $(V_0),(V_1)$ or it is bounded. In the first case, we borrow some ideas discussed in Bartsch--Wang \cite{bw} to prove compact embedding results of $X$ into the Orlicz spaces. When $V(x)$ is bounded, versions of ``vanishing lemma'' or ``Lions's Lemma'' play a very important role in order to obtain a nontrivial solution for Problem \eqref{p1}. The pioneer result was introduced by Lions in \cite{Lions1} for the standard Sobolev space. After that, many versions were introduced based on the class of operators involved in the problem, see for instance, \cite{Fan1,alves,Secchi}. We point out \cite{alves} where the authors have introduced a version of the ``vanishing Lemma'' for Orlicz--Sobolev spaces. However, up to now, a fractional version for Orlicz-Sobolev spaces was not considered. In general, the proof of these results are based on estimates involving derivatives and some interpolations inequalities. For instance, one can observe that $W^{1, \Phi}(\mathbb{R}^N)$ and $L_{\Phi}(\mathbb{R}^N)$ are interpolation spaces. In our case, for  fractional Orlicz--Sobolev space we are not able to apply the interpolation law.  Thus, we introduce new tricks and a fine analysis on the sequences $(u_n)$ in $W^{s, \Phi}(\mathbb{R}^N)$ that weakly converges to zero.

Summarizing, our work is divided into two parts:

\vspace{0,4cm}

	\noindent \textbf{Part I }-- In the first part, we develop \textit{classical abstract results} regarding to the fractional Orlicz--Sobolev spaces. Precisely, we obtain compact embedding results on the weighted fractional Orlicz--Sobolev space $X$ into the Orlicz space $L_{\Phi}(\mathbb{R}^{N})$, we study the operator $(-\Delta_{\Phi})^{s}+V(x)\varphi(\cdot)$ and we introduce a very useful version of the ``vanishing'' Lions' Lemma to the fractional Orlicz--Sobolev framework.

\vspace{0,4cm}

	\noindent \textbf{Part II} -- In the second part of our work we make use of the abstract results to give two applications regarding to existence of ground state solutions for Problem \eqref{p1}, taking into account unbounded or bounded potential $V(x)$. For this purpose, we use a variational approach based on Nehari method together with a fine analysis on the behavior of minimizing sequences. The main point is to restore the compactness using the abstract results obtained in Part I.

\vspace{0,4cm}

Now we state our abstract results regarding to the Part I. The first main result can be state as follows:

 \begin{theorem}[Compact embedding]\label{comp1}
	Assume that $(\varphi_1)-(\varphi_3)$ and $(V_{0})$--$(V_{1})$ hold. Then, the embedding $X\hookrightarrow L_\Phi(\mathbb{R}^N)$ is compact.
\end{theorem}

\begin{remark}\label{lambda1} We point out that Theorem \ref{comp1} implies that
	\begin{equation}\label{lambda11}
	\lambda_{1}:=\inf_{u\in X\backslash\{0\}}\left\{\frac{\displaystyle\iint \Phi\left(\frac{|u(x)-u(y)|}{|x-y|^{s}}\right)\frac{\mathrm{d}x\mathrm{d}y}{|x-y|^{N}}+\int V(x)\Phi(|u|)\;\mathrm{d}x}{\displaystyle\int \Phi(|u|)\;\mathrm{d}x}\right\}
	\end{equation}
	is attained. In particular, $\lambda_1 > 0$. On this subject, for bounded domains we refer the interested reader to \cite{Salort,Salort-1}.
\end{remark}

In order to state our second main result on compact embedding, we recall the property from \cite[Theorem 2]{rao}, precisely, for given $\Psi$ such that $\Psi<<\Phi_{*}$ with $\Psi \in \Delta_2$, there exists an $N$--function $R$, such that $\Psi\circ R<\Phi_{*}$. In this way, we can state our second main result as follows:

\begin{theorem}[Compact embedding]\label{compact}
	Assume that $(\varphi_{1})$--$(\varphi_{3})$ and $(V_{0})$--$(V_{1})$ hold. Suppose that $\Phi < \Psi<<\Phi_{*}$ and at least one of the following conditions hold:
	\begin{itemize}
		\item[(i)]  The following limit holds
		\begin{equation}\label{mla1}
		\limsup_{\vert t \vert\rightarrow0}\frac{\Psi(\vert t\vert)}{\Phi(\vert t\vert)}<+\infty.
		\end{equation}
		\item[(ii)] Suppose that $\Psi\in\Delta_{2}$ and there exists $b\in(0,1)$ such that
		\begin{equation}\label{mla3}
		\Psi(\widetilde{R}(|t|^{1-b})) \leq  C\Phi(|t|), \quad \vert t\vert\leq 1,
		\end{equation}
		where $\tilde{R}$ is the complementary function of $R$.
	\end{itemize}
	Then, the space $X$ is compactly embedded into $L_{\Psi}(\mathbb{R}^{N})$.
\end{theorem}

\begin{remark}
It is important to stress that in item $(i)$ for Theorem \ref{compact}  the limits \eqref{mla1} may not be zero. Furthermore, for item $(i)$ the $\Delta_2$ condition is not required for $\Psi$. This fact allows us to consider several $N$-functions $\Psi$. Notice also that assumption $\ell^{*}>m$ is used only in order to ensure the continuous embedding $X\hookrightarrow L_\Phi(\mathbb{R}^N)$. Here we refer the interested reader to \cite{Bahrouni-Emb} where some continuous embedding are proved. The item $(ii)$ for Theorem \ref{compact} is related with the fact that the interpolation law in our framework is not available. Hence we need to control the sequences in $W^{s,\Phi}(\mathbb{R}^N)$ proving that compact embedding for the fractional Orlicz-Sobolev spaces into the Orlicz spaces are verified.
\end{remark}

In order to prove the strong convergence of minimizing sequences in Part II, we prove the following result:
\begin{theorem}\label{homeo}
	Assume that $(\varphi_{1})$--$(\varphi_{3})$ and $(V_{0})$--$(V_{1})$ hold. Then the operator $(-\Delta_{\Phi})^{s}+V(x)\varphi(\cdot)$ is a homeomorphism.
\end{theorem}
The proof of Theorem \ref{homeo} relies on the theory of monotone operators together with the Browder--Minty Theorem, see \cite{Zeidler}. Finally, we obtain a useful version of the ``vanishing'' Lions' Lemma to the fractional Orlicz--Sobolev framework.

 \begin{theorem}[Lions' Lemma type result]\label{lions}
	Suppose that $(\varphi_1)$--$(\varphi_3)$ hold and
	\begin{equation}\label{mla1b}
	\lim_{|t|\to 0}\frac{\Psi(t)}{\Phi(t)}=0.
	\end{equation}
   Let $(u_{n})$ be a bounded sequence in $W^{s,\Phi}(\mathbb{R}^{N})$ in such way that $u_n \rightharpoonup 0$ in $X$ and
	\begin{equation}\label{lionss}
	\lim_{n\rightarrow+\infty}\left[\sup_{y\in\mathbb{R}^{N}}\int_{B_{r}(y)}\Phi(u_{n})\,\mathrm{d}x \right]=0,
	\end{equation}
	for some $r>0$. Then, $u_{n}\rightarrow0$ in $L_{\Psi}(\mathbb{R}^{N})$, where $\Psi$ is an $N$-function such that $ \Psi<<\Phi^{*}$.
\end{theorem}

Regarding to Part II, it is worthwhile to mention that $u \in W^{s,\Phi}(\mathbb{R}^N)$ is said to be a weak solution for the Problem \eqref{p1} whenever
\begin{equation*}
\iint\varphi\left(|D_s u|\right)|D_s u|  D_s h \,\mathrm{d}\mu+\int V(x)\varphi(|u|)uh\;\mathrm{d}x-\int f(x,u)h\;\mathrm{d}x=0, \quad \mbox{for all } h \in W^{s,\Phi}(\mathbb{R}^N),
\end{equation*}
where $\mathrm{d}\mu(x,y):=\frac{ \mathrm{d}x\mathrm{d}y}{|x-y|^N}$. In order to use a variational approach to obtain solutions, we suppose that $f\in C^{1}(\mathbb{R}^{N}\times\mathbb{R},\mathbb{R})$ satisfies the following assumptions:

\begin{enumerate}[label=($f_0$),ref=$(f_0)$]
	\item \label{f0}
	 there exist a function $\psi:[0,+\infty)\rightarrow[0,+\infty)$ and a constant $C>0$ such that
	  \[
	   |f(x,t)|\leq C\psi(t)t, \quad \mbox{for all} \hspace{0,2cm} (x,t)\in\mathbb{R}^{N}\times [0,+\infty),
	  \]
	 where $\Psi(t)=\int_{0}^{t}\psi(s)\;\mathrm{d}s$ is an N-function satisfying $\Psi<<\Phi_{*}$ (for definitions and properties see Section 2 ahead) and
	  \begin{equation}\label{psi1}
	   1<\ell\leq m<\ell_{\Psi}:=\inf_{t>0}\frac{t\psi(t)}{\Psi(t)}\leq \sup_{t>0}\frac{t\psi(t)}{\Psi(t)}=:m_{\Psi}<\ell^{*}:=\frac{\ell N}{N-\ell}.
	  \end{equation}
\end{enumerate}
\begin{enumerate}[label=($f_1$),ref=$(f_1)$]
	\item \label{f1}
	 $\displaystyle\lim_{|t|\rightarrow0}\frac{f(x,t)}{t\varphi(t)}=0$, uniformly in $x\in\mathbb{R}^{N}$.
\end{enumerate}

\begin{enumerate}[label=($f_2$),ref=$(f_2)$]
	\item \label{f3}
	$\displaystyle\lim_{|t|\rightarrow+\infty}\frac{f(x,t)}{|t|^{m-1}}=+\infty$, uniformly in $x\in\mathbb{R}^{N}$.
\end{enumerate}
\begin{enumerate}[label=($f_3$),ref=$(f_3)$]
	\item \label{f4}
	the map $t\mapsto\displaystyle\frac{f(x,t)}{|t|^{m-2}t}$ is strictly increasing for $t>0$ and strictly decreasing for $t<0$.
\end{enumerate}
\begin{enumerate}[label=($f_4$),ref=$(f_4)$]
	\item \label{f5}
	There exists $\theta>m$ such that $\theta F(x,t):=\theta\int_{0}^{t}f(x,s)\,\mathrm{d}s\leq t f(x,t)$, for $(x,t)\in\R^N\times\R.$
\end{enumerate}

\begin{remark}
	We point out that assumption \ref{f0} implies that $$\displaystyle\lim_{|t|\rightarrow+\infty}\frac{f(x,t)}{t\varphi_*(t)}=0$$
	holds uniformly in $x\in\mathbb{R}^{N}$ where $t\varphi_*(t)$ is an N-function. In fact, this function  is the Sobolev conjugate function $\Phi_*$ of $\Phi$, see Section \ref{s0}. In particular, we have that $\Phi_*(t)=\int_{0}^{t}\tau\varphi_*(\tau)\,\mathrm{d}\tau$ holds for each $t \geq 0$.
\end{remark}

It is important to emphasize that hypothesis $(V_1)$ is related to the existence of compact embedding for the Orlicz-Sobolev spaces into Orlicz spaces.
 As a consequence, we obtain that
 \begin{equation}\label{lambda1}
 \lambda_{1}:=\inf_{u\in X\backslash\{0\}}\left\{\frac{\displaystyle\iint \Phi\left(\frac{|u(x)-u(y)|}{|x-y|^{s}}\right)\frac{\mathrm{d}x\mathrm{d}y}{|x-y|^{N}}+\int V(x)\Phi(|u|)\;\mathrm{d}x}{\displaystyle\int \Phi(|u|)\;\mathrm{d}x}\right\}>0.
\end{equation}
 In fact, by using the compact embedding proved in Theorem \ref{comp1} or Theorem \ref{lions}, taking any minimizer sequence for \eqref{lambda1} we prove that $\lambda_ 1$ is attained. In particular, we know that $\lambda_1 > 0$. On this subject for bounded domains we refer the interested reader to \cite{Salort,Salort-1}.

In order to apply the Nehari method in our framework, we also suppose the following hypothesis:

\begin{itemize}
	\item[($\varphi_4$)] There holds
	\begin{equation}\label{emj2}
	\ell-2\leq\frac{t\varphi^{\prime}(t)}{\varphi(t)}\leq m-2<\ell^*-2.
	\end{equation}
\end{itemize}
We point out that \eqref{emj2} implies that hypothesis $(\varphi_3)$ is satisfied.

Now, we are in a condition to state our first result of Part II.

\begin{theorem}[Unbounded potential] \label{A} Assume that the function $\varphi$ satisfies $(\varphi_1),(\varphi_2)$ and $(\varphi_4)$. Suppose that $V(x)$ verifies $(V_0),(V_1)$ and $f$ satisfies hypotheses \ref{f0}--\ref{f5}. Then, Problem \eqref{p1} admits at least one ground state solution $u \in W^{s,\Phi}(\mathbb{R}^N)$.	
\end{theorem}

Now we shall study the case when $V(x)$ is bounded. For this purpose, we consider the following hypothesis:
\begin{enumerate}[label=($V_2$),ref=$(V_2)$]
	\item \label{(V_2)}
Assume that $x \mapsto V(x)$ and $x \mapsto f(x,u)$ are $1$-periodic functions.
\end{enumerate}

Finally, we state the existence result regarding bounded potential.

\begin{theorem} [Bounded potential] \label{B}
Assume that the function $\varphi$ satisfies $(\varphi_1),(\varphi_2)$ and $(\varphi_4)$. Suppose that $V(x)$ verifies $(V_0), (V_2)$ and $f$ satisfies hypotheses \ref{f0}-\ref{f5}. Then, Problem \eqref{p1} admits at least one ground state solution $u \in W^{s,\Phi}(\mathbb{R}^N)$.	
\end{theorem}

It is worthwhile to mention that for $\Phi(t)=|t|^2(\ln(1+|t|))^q,~ 2+q<2^*$ we deduce that embedding $X\hookrightarrow L_\Psi(\mathbb{R}^N)$ is compact for each $\Psi$ such that $\Psi << \Phi^*$. Similarly, we apply Theorem \ref{lions} proving that $u_{n}\rightarrow0$ in $L_{\Psi}(\mathbb{R}^{N})$ for all $N$-function $\Psi$ with $\Psi < < \Phi^*$ such that \eqref{mla1b} is satisfied. Notice also that the Orlicz space given by the $N$-function $\Phi(t)=|t|^2(\ln(1+|t|))^q$ does not coincide with any Lebesgue space $L^p(\mathbb{R}^N)$ with $p \in [1, \infty)$. Hence for our setting Lebesgue spaces $L^p(\mathbb{R}^N)$ are not sufficient. Furthermore, considering the fractional $\Phi$-Laplace operator defined in the whole space $\mathbb{R}^N$, we can prove our main results taking into account the Orlicz-Sobolev framework.

The remainder of this work is organized as follows: In the forthcoming Section we introduce some preliminary concepts and auxiliary results that will be useful throughout the work. In Section \ref{s3} we give the proofs of the abstract results. In Section \ref{s4} we prove the existence results regarding Problem \eqref{p1}.

\section{Preliminary results}

\subsection{Fractional Orlicz--Sobolev spaces}\label{s0}

In this section we introduce some preliminary concepts about fractional Orlicz-Sobolev spaces. For a more complete discuss about this subject we refer the readers to \cite{Bonder}.

%In this Section we recall some preliminary concepts about the Orlicz-Sobolev spaces which will be used throughout the paper. For a more complete discussion about on this subject we refer the readers to \cite{adams,rao}.
Let $\Phi:\mathbb{R}\rightarrow[0,+\infty)$ be convex and continuous. We recall that $\Phi$ is an $N$-function if satisfies the following conditions:
\begin{itemize}
	\item[(i)] $\Phi$ is even;
	\item[(ii)] $\displaystyle\lim_{t\rightarrow 0}\frac{\Phi(t)}{t}=0$;
	\item[(iii)] $\displaystyle\lim_{t\rightarrow \infty}\frac{\Phi(t)}{t}=\infty$;
	\item[(iv)] $\Phi(t)>0$, for all $t>0$.	
\end{itemize}
Notice that by using assumptions $(\varphi_1)$ and $(\varphi_2)$ we conclude that $\Phi$ is an $N$-function. We say that an $N$-function satisfies the $\Delta_{2}$-condition if there exists $K>0$ such that
 \[
  \Phi(2t)\leq K\Phi(t), \quad \mbox{for all} \hspace{0,2cm} t\geq0.
 \]

Throughout the work we use the following notation:
 $$
 \xi^{-}(t)=\min\{t^\ell,t^m\} \quad \mbox{and} \quad \xi^{+}(t)=\max\{t^\ell,t^m\}, \quad t\geq 0.
 $$
We recall the following Lemma due to N. Fukagai et al. \cite{Fuk_1} which can be written in the following form:
\begin{proposition}\label{lema_naru}
	Assume that $(\varphi_1)-(\varphi_3)$ hold. Then, $\Phi$ satisfies the following estimates:
	$$
	\xi^{-}(t)\Phi(\rho)\leq\Phi(\rho t)\leq \xi^{+}(t)\Phi(\rho), \quad \rho, t> 0,
	$$
	$$
	\xi^{-}(\|u\|_{\Phi})\leq\int_\Omega\Phi(u)\,\mathrm{d}x\leq \xi^{+}(\|u\|_{\Phi}), \quad u\in L_{\Phi}(\Omega).
	$$
\end{proposition}

We consider the spaces
\begin{align*}
&L_\Phi(\Omega) :=\left\{ u\colon \Omega \to \R \text{ measurable  such that }  \rho_{\Phi,\Omega}(u) < \infty \right\},\\
&W^{s,\Phi}(\Omega):=\left\{ u\in L_\Phi(\Omega) \text{ such that } \rho_{s,\Phi,\R^N}(u)<\infty \right\},\\
&W^{s,\Phi}_0(\Omega):=\left\{ u\in W^{s,\Phi}(\R^n) :\ u=0\ a.e.\ \text{in}\ \R^N\setminus\Omega\right\},
\end{align*}
where the modulars $\rho_{\Phi,\Omega}$ and $\rho_{s,\Phi}$ are defined as
\begin{align*}
&\rho_{\Phi,\Omega}(u):=\int_{\Omega} \Phi(|u(x)|)\,\mathrm{d}x,\\
&\rho_{s,\Phi,\R^N}(u):=
  \int_{\R^N\times\R^N} \Phi( |D_su(x,y)|)  \,\mathrm{d}\mu,
\end{align*}
and $\mathrm{d}\mu(x,y):=\frac{ \mathrm{d}x\mathrm{d}y}{|x-y|^N}$.
These spaces are endowed with the so-called \emph{Luxemburg norms}
\begin{align*}
&\|u\|_{L_\Phi(\Omega)} := \inf\left\{\lambda>0\colon \rho_{\Phi,\Omega}\left(\frac{u}{\lambda}\right)\le 1\right\},\\
&\|u\|_{W^{s,\Phi}(\Omega)} := \|u\|_{L^\Phi(\Omega)} + [u]_{W^{s,\Phi}(\R^N)},
\end{align*}
where the  {\em $(s,\Phi)$-Gagliardo semi-norm} is defined as
\begin{align*}
&[u]_{W^{s,\Phi}(\R^N)} :=\inf\left\{\lambda>0\colon \rho_{s,\Phi,\R^N}\left(\frac{u}{\lambda}\right)\le 1\right\}.
\end{align*}
Since we assume $(\varphi_3)$, $W^{s,\Phi}(\Omega)$ is a reflexive Banach space. Moreover $C_c^\infty(\mathbb{R}^{N})$ is dense in $W^{s,\Phi}(\R^N)$. On this subject see \cite[Proposition 2.11]{Bonder} and \cite[Proposition 2.9]{DNFBS}.

In order to state some embedding results for fractional Orlicz-Sobolev spaces we introduce the following notation:

Given two $N$-functions $A$ and $B$, we say that \emph{$B$ is essentially stronger than $A$} or equivalently that \emph{$A$ decreases essentially more rapidly than $B$}, and denoted by $A<< B$, if for each $a>0$ there exists $x_a\geq 0$ such that $A(x)\leq B(ax)$ for $x\geq x_a$. Similarly, we write $A < B$ whenever $A(x) \leq B(x)$ for each $x \geq x_0$ with $x_0 > 0$ fixed.

When the $N$-function $\Phi$ fulfills condition $(\varphi_4)$, the critical function for the fractional Orlicz-Sobolev embedding is given by
$$
\Phi_{*}^{-1}(t)=\int_{0}^{t}\frac{G^{-1}(\tau)}{\tau^{\frac{N+s}{N}}}\mathrm{d}\tau.
$$

The following result can be found in \cite{Bahrouni-Emb}. On this subject see also \cite{Cianchi} for further generalizations.
\begin{theorem}\label{ceb}
	Let $\Phi$ be an $N$-function satisfying
$$
\displaystyle\int_{0}^{1}\frac{\Phi^{-1}(\tau)}{\tau^{\frac{N+s}{N}}}\mathrm{d}\tau<\infty\quad \text{and}\quad
\displaystyle\int_{1}^{+\infty}\frac{\Phi^{-1}(\tau)}{\tau^{\frac{N+s}{N}}}\mathrm{d}\tau=\infty,
$$
where $0<s<1$ and $\Omega\subset \R^N$ is a $C^{0,1}$ bounded open subset. Then, the following statements hold:
\begin{itemize}
  \item[(i)] \label{7}
  the embedding $W^{s,\Phi}(\Omega)\hookrightarrow L^{\Phi_{*}}(\Omega)$ is continuous;

  \item[(ii)] \label{Bem}
  for any $N$-function $B$ such that $B << \Phi_{*}$, the embedding  $W^{s,\Phi}(\Omega)\hookrightarrow L_{B}(\Omega)$ is compact.
\end{itemize}
\end{theorem}

\begin{lemma}\label{V0ineq}[Lemma 3.1 \cite{Bahrouni}]
	Let $\Omega \subseteq \mathbb{R}^N$ and consider an $N$-function $\Phi$ satisfying $(\varphi_3)$. Then,
	$$\xi^-([u]_{s,\Phi})\leq \iint \Phi\left(|D_s u|\right)\,\mathrm{d}\mu\leq \xi^+([u]_{s,\Phi}),\quad u\in W^{s,\Phi}(\Omega).$$
\end{lemma}

 \begin{lemma}\label{V.ineq}[Lemma 4.3 \cite{Bahrouni-Emb}]
 	Let $\Phi$ be an $N$-function satisfying $(\varphi_3)$. If $V(x)$ satisfies $(V_1)$ and $(V_2)$, then
 	$$\xi^-(\|u\|_{V,\Phi})\leq \int_{\mathbb{R}^N}V(x)\Phi(u)\mathrm{d}x\leq \xi^+(\|u\|_{V,\Phi}),\quad u\in X.$$
 \end{lemma}

\subsection{Technical lemmas}

We start this subsection by proving a result inspired by \cite[Theorem 2, p.16]{rao}. This result plays a very important role in order to obtain compact embedding taking into account some hypothesis on the potential $V$. 

\begin{lemma}\label{Aux}
	Let $M, N$ be N-functions such that $M<<N$. Then, there exists an N-function $R$ satisfying $R\circ M< N$.
\end{lemma}
\begin{proof}
	Let $r:[0,+\infty)\rightarrow\mathbb{R}$ be defined by
	$$
	r(t):=\left\{
	\begin{array}{ll}
	\displaystyle\inf_{x>t}\frac{N\circ M^{-1}(x)}{x},& t\geq 1\\
	r(1)t,& 0\leq t<1.
	\end{array}
	\right.
	$$
	It is not hard to verify that $r(t)$ satisfies the following properties:
	\begin{enumerate}
		\item[$(i)$] $\displaystyle\lim_{t\to \infty} r(t)=\infty$;
		\item[$(ii)$] $\displaystyle\lim_{t\to 0} r(t)=0$;
		\item[$(iii)$] $t\mapsto r(t)$ is non increasing.
	\end{enumerate}
	Let $R(t):=\int_0^tr(s)\,\mathrm{d}s$ be the primitive of $r(t)$. We consider the extension of $R(t)$ to $\mathbb{R}^-$ by setting $R(t)=R(-t)$. One may conclude that $R$ is an $N$-function. Finally, it follows that
	$$R(t)\leq r(t)t=N\circ M^{-1}(t), \quad t\geq 1,$$
	which is equivalent to $R\circ M(s)\leq N(s),~s\geq M^{-1}(1)$.	
\end{proof}

\begin{remark}\label{obs2}
	By choosing $M=\Phi$ and  $N=\Phi_*$ in Lemma \ref{Aux}, we obtain that $R\circ \Phi<\Phi_*$. Thus,
$
X\hookrightarrow L_{R\circ\Phi}(\mathbb{R}^{N}),
$
that is, there exists $S_{R\circ\Phi}>0$ such that
		\[
		\Vert u\Vert_{R\circ\Phi}\leq S_{R\circ\Phi}\Vert u\Vert.
		\]
	%In fact, we have the inclusions $X\subset W^{s,\Phi}(\mathbb{R}^{N})\subset L^{R\circ\Phi}(\mathbb{R}^{N})$. In light of \cite[Theorem 6.1.17]{Fucik}, we conclude that .
\end{remark}

\begin{lemma}\label{interp11}
	Let $\Psi,R$ be N-functions. Then
	$$\Psi\left(\frac{xy}{4}\right)\leq \Psi(R(x))+\Psi(\widetilde{R}(y)),~x,y\geq 0.$$
\end{lemma}
\begin{proof}
	Using Young's inequality and the fact that $R,\widetilde{R},\Psi$ are convex functions, we have
	\begin{eqnarray}
	\Psi(\frac{xy}{4})&\leq& \Psi\left(R(\frac{x}{2})+\widetilde{R}(\frac{y}{2})\right)\nonumber\\
	&\leq& \frac12\Psi(R({x}))+\frac12\Psi(\widetilde{R}({y}))\nonumber\\
	&\leq& \Psi(R({x}))+\Psi(\widetilde{R}({y})).\nonumber
	\end{eqnarray}
This ends the proof.
\end{proof}

\begin{remark}\label{rem-r}
	It is important to point out that in Lemma \ref{Aux}, we can redefine $r$ near to the origin in such way that $R$ is an $N$--function. For instance, given $b\in(0,1)$ we can redefine $r$ by
	\[
	r(t)=\frac{br(1)}{\Phi_*'(1)}\Phi_*'(t^{1/b})t^{1/b-1},\quad t<1.
	\]	
\end{remark}

%%%%%%%%%%%%%%%%%%%%%%%%%%%%%%%%%%%%%%%%%%%%%%%%%%%%%%%%%%%%%%%%%%%%%%%%%%%%%%%%%%%%%%%%%%%%%%%%%%%%%%%%%%%%%%%%%%%%%%%%%%%%%%%%%%%%%%%%%%%%%%%%%%%%%%%%%%%%%%%%%%%%%%%%%%%%%%%%%%%%%%%
%                                                                             THE VARIATIONAL FRAMEWORK
%%%%%%%%%%%%%%%%%%%%%%%%%%%%%%%%%%%%%%%%%%%%%%%%%%%%%%%%%%%%%%%%%%%%%%%%%%%%%%%%%%%%%%%%%%%%%%%%%%%%%%%%%%%%%%%%%%%%%%%%%%%%%%%%%%%%%%%%%%%%%%%%%%%%%%%%%%%%%%%%%%%%%%%%%%%%%%%%%%%%%%%

\section{Part I -- Proof of the abstract results}\label{s3}

\subsection{Proof of Theorem \ref{comp1} (Compact embedding)}

Let $(u_{n})\subset X$ be such that $u_{n}\rightharpoonup u$ weakly in $X$. Thus, we have that $u_{n}\rightarrow u$ strongly in $L_{\Psi,loc}(\mathbb{R}^{N})$ with $\Phi<\Psi<<\Phi_{*}$. Let us prove that $u_{n}\rightarrow u$ strongly in $L_{\Phi}(\mathbb{R}^{N})$. In view of Br\'ezis-Lieb's Theorem, it suffices to show that
 \[
  \alpha_{n}:=\int_{\mathbb{R}^{N}}\Phi(u_{n})\rightarrow\int_{\mathbb{R}^{N}}\Phi(u).
 \]
Since $(\alpha_{n})$ is bounded we have, up to a subsequence, $\alpha_{n}\rightarrow \alpha$. Hence, it follows from Fatou's Lemma that $\alpha\geq\int_{\mathbb{R}^N}\Phi(u)\,\mathrm{d} x$. By using local convergence and that $\Phi\in\Delta_{2}$ we have
 \begin{equation}\label{c}
  \int_{B_{R}}\Phi(u_{n})\rightarrow \int_{B_{R}}\Phi(u),
 \end{equation}
where $B_{R}:=B_{R}(0)$.

\vspace{0,3cm}

\noindent \textit{Claim.} For given $\varepsilon>0$, there exists $R=R(\varepsilon)>0$ such that
 \begin{equation}\label{c1}
  \int_{B_{R}^{c}}\Phi(u_{n})<\varepsilon,
 \end{equation}
for $n\in\mathbb{N}$ large.

\vspace{0,3cm}

If \eqref{c1} holds, then
 \begin{eqnarray*}
  \int_{\mathbb{R}^{N}}\Phi(u) & = & \int_{B_{R}}\Phi(u)+\int_{B_{R}^{c}}\Phi(u)\\
                                                & \geq & \lim_{n\rightarrow\infty}\int_{B_{R}}\Phi(u_{n})\\
                                                & = & \lim_{n\rightarrow\infty}\int_{\mathbb{R}^{N}}\Phi(u_{n})-\lim_{n\rightarrow\infty}\int_{B_{R}^{c}}\Phi(u_{n})\\
                                                & \geq & \alpha-\varepsilon.
 \end{eqnarray*}
Now, let us prove the \textit{Claim.} For given $\varepsilon>0$, let $M>0$ be such that
 \begin{equation}\label{ml1}
  \frac{2}{\varepsilon}\max\{B_1,\sup_n\xi_1(\|u_n\|)\}<M,
 \end{equation}
where $B_1$ will be defined below.  By choosing $M=\Phi$ and  $N=\Phi_*$ in Lemma \ref{Aux}, we obtain that $R\circ \Phi<\Phi_*$. Define
$$f(t):=tR^{-1}\left(\frac{1}{t}\right), \quad t>0.$$
One may check that $f$ is continuous and $\displaystyle\lim_{t\to 0}f(t)=0$. In view of $(V_{1})$, let $R>1$ be sufficiently large such that
 \begin{equation}\label{ml2}
  f(\vert \{x\in B_{R}^{c}:V(x)<M \} \vert)\leq \frac{\epsilon}{2B_1}.
 \end{equation}

 Let us define the sets
 \[
  \mathcal{A}:=\{x\in B_{R}^{c}:V(x)\geq M \}
 \quad \mbox{and} \quad
 \mathcal{B}:=\{x\in B_{R}^{c}:V(x)<M \}.
 \]
In view of \cite[Lemma 4.3]{Bahrouni-Emb} we have
 \begin{equation*}
  \int_{\mathcal{A}}\Phi(u_{n}) \leq  \int_{\mathcal{A}}\frac{V(x)}{M}\Phi(u_{n})\\
                                                \leq  \frac{1}{M}\xi_{1}(\Vert u_{n}\Vert_{V,\Phi})\\
                                                \leq  \frac{1}{M}\xi_{1}(\Vert u_{n}\Vert)<\frac{\varepsilon}{2}.
 \end{equation*}
On the other hand, it follows from Lemma \ref{Aux} and H\"{o}lder's inequality that
 \begin{eqnarray}\label{ineq1}
  \int_{\mathcal{B}}\Phi(u_{n}) & \leq & 2\|\Phi(u_n)\|_R\| \chi_{_\mathcal{B}}\|_{\widetilde{R}},
 \end{eqnarray}
where $\tilde{R}$ is the complementary function of $R$. But, from \cite[Proposition 4.6.9]{Fucik}
\begin{eqnarray}\label{ineq2}
	 	\| \chi_{_\mathcal{B}}\|_{\widetilde{R}}\leq C_0 |\mathcal{B}|R^{-1}\left(\frac{1}{|\mathcal{B}|}\right)=C_0f(|\mathcal{B}|).
\end{eqnarray}
Moreover, since $R\circ\Phi<\Phi_*$, there exist $C_1,T>0$ such that $R(\Phi(t))\leq C_1\Phi_*(t)$ for all $t\geq T$. Consequently,
\begin{eqnarray}\label{ineq3}
\int_{\mathcal{B}}R(\Phi(u_n))dx&=&    \left(\int_{\mathcal{B}\cap[|u_n|\leq T]}+\int_{\mathcal{B}\cap[|u_n|> T]}\right)R(\Phi(u_n))\nonumber\\
                                &\leq& |\mathcal{B}|R(\Phi(T))+\sup_n\int_{\mathbb{R}^N} \Phi_*(u_n)\,\mathrm{d}x\nonumber\\
                                &\leq& |\mathcal{B}_1|R(\Phi(T))+\sup_n\int_{\mathbb{R}^N} \Phi_*(u_n)\,\mathrm{d}x
\end{eqnarray}
where $\mathcal{B}\subset \mathcal{B}_1$ for all $0<\epsilon<1$ small enough. Here, we define
$$B_1:=2C_0|\mathcal{B}_1|R(\Phi(T))+2C_0\sup_n\int_{\mathbb{R}^N} \Phi_*(u_n)\,\mathrm{d}x.$$
Thus, by \eqref{ineq1}-\eqref{ineq3} and \eqref{ml2}, we obtain
$$
\int_{\mathcal{B}}\Phi(u_{n})\leq B_1 |\mathcal{B}|R^{-1}\left(\frac{1}{|\mathcal{B}|}\right)=B_1 f(|\mathcal{B}|)\leq \frac{\varepsilon}{2}.
$$
Therefore,
 \[
  \int_{B_{R}^{c}}\Phi(u_{n})=\int_{\mathcal{A}}\Phi(u_{n})+\int_{\mathcal{B}}\Phi(u_{n})<\varepsilon,
 \]
which implies the \textit{Claim} and finishes the proof.

\subsection{Proof of Theorem \ref{compact} (Compact embedding)}

  Since $\Psi<<\Phi_{*}$, for given $\varepsilon>0$, there exists $T>0$ such that
	 \begin{equation}\label{ml4}
	  \frac{\Psi(\vert t\vert)}{\Phi_{*}(\vert t\vert)}\leq \frac{\varepsilon}{2\kappa}, \quad \vert t\vert \geq T,
	 \end{equation}
	where $\kappa>0$ will be chosen later. Let $(u_{n})\subset X$ be such that $u_{n}\rightharpoonup 0$ weakly in $X$. In view of Proposition \ref{comp1}, it follows that $u_{n}\rightarrow0$ strongly in $L_{\Phi}(\mathbb{R}^{N})$. Note that
	 \begin{equation}\label{ml5}
	  \int_{\mathbb{R}^{N}}\Psi(\vert u_{n}\vert)\,\mathrm{d}x=\int_{\{\vert u_{n}\vert\geq T\}}\Psi(\vert u_{n}\vert)\,\mathrm{d}x+\int_{\{\vert u_{n}\vert<T\}}\Psi(\vert u_{n}\vert)\,\mathrm{d}x.
	 \end{equation}
	
	  Let us set
	  \[
	   \kappa=\sup_n\int_{\mathbb{R}^{N}}\Phi_{*}(\vert u_{n}\vert)\,\mathrm{d}x.
	  \]
	 It follows from \eqref{ml4} that
	  \begin{equation}\label{ml6}
	   \int_{\{\vert u_{n}\vert\geq T\}}\Psi(\vert u_{n}\vert)\,\mathrm{d}x\leq \frac{\varepsilon}{2\kappa}\int_{\mathbb{R}^{N}}\Phi_{*}(\vert u_{n}\vert)\,\mathrm{d}x\leq \frac{\varepsilon}{2}.
	  \end{equation}
	 Now, we shall use assumption \eqref{mla1} or \eqref{mla3} to study the integral in \eqref{ml5} when $\vert u_{n}\vert< T$.
	
	  \subsubsection{\textbf{Proof of Theorem \ref{compact} assuming \eqref{mla1}}}
	
	  Given $a\in(0,1)$, it follows from H\"{o}lder's inequality that
	  \begin{equation}\label{ml9}
	   \int_{\{\vert u_{n}\vert<T\}}\Psi(\vert u_{n}\vert)\,\mathrm{d}x\leq \left[\int_{\{\vert u_{n}\vert<T\}}\left(\frac{\Psi(\vert u_{n}\vert)}{\Phi(\vert u_{n}\vert)^{a}} \right)^\frac{1}{1-a}\,\mathrm{d}x \right]^{1-a}\left[\int_{\mathbb{R}^{N}}\Phi(\vert u_{n}\vert)\,\mathrm{d}x \right]^{a}.
	  \end{equation}
	 \vspace{0,3cm}
	
	 \noindent \textit{Claim.} There exists $\tilde{C}>0$ such that
	  \begin{equation}\label{ml7}
	   \left(\frac{\Psi(\vert u\vert)}{\Phi(\vert u\vert)^{a}} \right)^\frac{1}{1-a}\leq \tilde{C}\Phi(\vert u\vert), \quad \vert u\vert\leq T.
	  \end{equation}
	
	 \vspace{0,3cm}
	
	  In fact, in view of assumption \eqref{mla1}, there exists $\delta>0$ such that $\Psi(\vert u\vert)\leq C\Phi(\vert u\vert)$, for all $\vert u\vert \leq \delta$. If $\delta<T$, then
	   \[
	    \frac{\Psi(\vert u\vert)}{\Phi(\vert u\vert)}\leq \frac{\Psi(T)}{\Phi(\delta)}, \quad \vert u\vert\in[\delta,T].
	   \]
	  By defining $\tilde{C}^{1-a}:=\max\left\{C, \frac{\Psi(T)}{\Phi(\delta)}\right\}$ we conclude that
	   \[
	    \frac{\Psi(\vert u\vert)}{\Phi(\vert u\vert)}\leq \tilde{C}^{1-a}, \quad \vert u\vert\leq T,
	   \]
	  which implies \eqref{ml7}.
	
	  Finally, since $u_{n}\rightarrow0$ strongly in $L_{\Phi}(\mathbb{R}^{N})$, there exists $n_{0}\in\mathbb{N}$ such that
	   \begin{equation}\label{ml8}
	    \int_{\mathbb{R}^{N}}\Phi(\vert u_{n}\vert)\,\mathrm{d}x<\left(\frac{\varepsilon}{2\tilde{C}^{1-a}} \right)^{\frac{1}{a}}, \quad \mbox{for all } n>n_{0}.
	   \end{equation}
	  Therefore, it follows from \eqref{ml6}--\eqref{ml8} that
	   \begin{equation}\label{mla8}
	    \int_{\mathbb{R}^{N}}\Psi(\vert u_{n}\vert)\,\mathrm{d}x<\varepsilon.
	   \end{equation}
	   This finishes the proof.

\subsubsection{\textbf{Proof of Theorem \ref{compact} assuming \eqref{mla3}}}
	Now, we shall prove that
	\begin{equation}\label{eq1c1}
	\lim_{n\to 0}\int_{\{|u_n|< 1\}}\Psi(\vert u_{n}\vert)\,\mathrm{d}x =0.
	\end{equation}
	Recall that $\Psi$ satisfies the $\Delta_2$ condition. Hence, by Proposition \ref{lema_naru} one has
	$$\lim_{n\to 0}\int_{\{|u_n|< 1\}}\Psi(\vert u_{n}\vert)\,\mathrm{d}x \leq\max\{\|u_n\chi_{_{[|u_n|<1]}}\|_\Psi^{\ell_\Psi},\|u_n\chi_{_{[|u_n|<1]}}\|_\Psi^{m_\Psi}\}.$$
	In this case, to prove \eqref{eq1c1} it is sufficient to prove that
	$$\lim_{n\to \infty}\|u_n\chi_{_{[|u_n|<1]}}\|_\Psi=0.$$
	In this direction, considering N-function $R$ defined in Lemma \ref{interp11} and using \cite[Theorem 1, page 179]{rao}, we deduce
	\begin{eqnarray}\label{eq1d1}
	\|u_n\chi_{_{[|u_n|<1]}}\|_\Psi&=&\|(u_n\chi_{_{[|u_n|<1]}})^{b}(u_n\chi_{_{[|u_n|<1]}})^{1-b}\|_\Psi\nonumber\\
	                                                    &\leq& \|(u_n\chi_{_{[|u_n|<1]}})^{b}\|_{\Psi\circ R}\|(u_n\chi_{_{[|u_n|<1]}})^{1-b}\|_{\Psi\circ \widetilde{R}}.
	\end{eqnarray}
	But, in view of Remark \ref{rem-r}, by definition of $R$ and $X\hookrightarrow L_{\Phi_*}(\mathbb{R}^N)$, there exist $C_0,C_1>0$ such that
	\begin{eqnarray}
	\int_{\mathbb{R}^N}\Psi(R((u_n\chi_{_{[|u_n|<1]}})^{b}))\,\mathrm{d}x \leq C_0\int_{\{|u_n|<1\}}\Phi_*(|u_n|)\,\mathrm{d}x<C_1,
	\end{eqnarray}
	this implies that $(\chi_{_{[|u_n|<1]}}u_n^{b})$ is bounded in $L_{\Psi\circ R}(\mathbb{R}^N)$. Now, using assumption \eqref{mla3} we have
	\begin{eqnarray}
	\Psi(\widetilde{R}(\chi_{_{[|u_n|<1]}}|u_n|^{1-b})) \leq  C\Phi(|u_n|),\nonumber
	\end{eqnarray}
	where $C:=\frac{\Psi(\widetilde{R}(1))}{\Phi(1)}$.
	
\subsection{Proof of Theorem \ref{homeo}}\label{sec5}

For the sake of simplicity we denote $L:=(-\Delta_{\Phi})^{s}+V(x)\varphi(\cdot)$. It is not hard to see that $L$ is a continuous operator. Note that $L$ is bijective if and only if for given $f\in X^{*}$ the problem
 \begin{equation}\label{m1}
  L(u)=(-\Delta_{\Phi})^{s}u+V(x)\varphi(u)u=f(x), \quad \mbox{in } \mathbb{R}^{N},
 \end{equation}
admits a unique solution $u\in X$. In order to do that we shall apply the Browder-Minty Theorem \cite{deimling, Zeidler} finding at least one solution for the problem \eqref{m1}. In fact, we claim that the operator $L$ is coercive. Firstly, consider the Luxemburg's norm given by
	$$\|u\|_1:=\inf\left\{\lambda>0:\iint \Phi\left(\dfrac{\vert D_{s}u\vert}{\lambda} \right)\,\mathrm{d}\mu+\int V(x)\Phi\left(\dfrac{u}{\lambda} \right)\,\mathrm{d}x\leq 1 \right\}.$$
	It is not hard to prove that $\|u\|\leq 2\|u\|_1$ and
	$$\iint \Phi\left(\dfrac{\vert D_{s}u\vert}{\|u\|_1} \right)\,\mathrm{d}\mu+\int V(x)\Phi\left(\dfrac{u}{\|u\|_1} \right)\,\mathrm{d}x=1.$$
Under these conditions we obtain
 \begin{eqnarray}\label{m2}
  \dfrac{\langle L(u), u \rangle}{\Vert u\Vert} & = & \dfrac{1}{\Vert u\Vert}\left[\iint  \varphi(\vert D_{s}u\vert)\vert D_{s}u\vert^{2}\,\mathrm{d}\mu+\int V(x)\varphi(u)u^{2}\,\mathrm{d}x \right]\nonumber\\
              & \geq & \dfrac{\ell}{\Vert u\Vert}\left[\iint \Phi(\vert D_{s}u\vert )\,\mathrm{d}\mu +\int V(x)\Phi(u)\,\mathrm{d}x \right]\nonumber\\
                      & = & \dfrac{\ell}{\Vert u\Vert}\left[\iint \Phi\left(\dfrac{\vert D_{s}u\vert}{\Vert u\Vert}\Vert u\Vert \right)\,\mathrm{d}\mu+\int V(x)\Phi\left(\dfrac{u}{\Vert u\Vert}\Vert u\Vert \right)\,\mathrm{d}x  \right]\nonumber\\
                                  & \geq & \dfrac{\ell}{\Vert u\Vert}\max\{\Vert u\Vert^{\ell},\Vert u\Vert^{m} \}\left[\iint \Phi\left(\dfrac{\vert D_{s}u\vert}{2\Vert u\Vert_1} \right)\,\mathrm{d}\mu+\int V(x)\Phi\left(\dfrac{u}{2\Vert u\Vert_1} \right)\,\mathrm{d}x \right]\nonumber\\
                                           & = & \frac{\ell}{2^m} \max\{\Vert u\Vert^{\ell-1},\Vert u\Vert^{m-1} \}\rightarrow +\infty, \quad \mbox{as } \|u\| \rightarrow \infty.
 \end{eqnarray}
Here it was used the fact that $1<\ell \leq m$. Throughout this work $\langle \cdot,\cdot\rangle$ denotes the dual pair between $X^*$ and $X$. Furthermore, the operator $L$ is monotone, that is, we have that $\langle L(u) - L(v), u - v \rangle \geq 0$ holds true for each $u, v \in X$. Indeed, by using hypothesis $(\varphi_2)$ we obtain that the function $t \mapsto \Phi(t)$ is convex. In particular, we obtain that the functional $R: X \to \mathbb{R}$ given by
\[
R(u):=\iint\Phi(\vert D_{s}u\vert)\,\mathrm{d}\mu+\int V(x)\Phi(u)\,\mathrm{d}x
\]
is convex. Clearly, the functional $R$ is in $C^1$ class and taking into account the last assertion we obtain that $L$ is monotone. Hence the continuous operator $L$ is coercive and monotone proving that problem \eqref{m1} has at least one solution for each $f \in X^*$. For the uniqueness we apply also the theory of monotone operators, see \cite{Zeidler}. More precisely, one may check that $L$ is strictly monotonic, that is,  $\langle L(u) - L(v), u - v \rangle > 0$ holds true for each $u, v \in X$ such that $u \neq v$. This can be done thanks to hypothesis $(\varphi_2)$ which says that
$t \mapsto t \varphi(t)$ is strictly increasing. Hence the functional $R$ given above is strictly convex.  As a consequence, assuming that $u_1, u_2$ are solutions for the problem \eqref{m1}, we obtain
\begin{equation}
\langle L(u_1) - L(u_2), u_1 - u_2 \rangle = \langle L(u_1), u_1 - u_2 \rangle - \langle L(u_2), u_1 - u_2 \rangle = 0.
\end{equation}
Using the fact that $L$ is strictly monotonic we conclude that $u_1 = u_2$ in $X$.
Therefore, problem \eqref{m1} admits a unique solution for each $f \in X^*$. Define the operator $S: X^* \rightarrow X$ given by $S(f) = u$ where $f \in X^*$ and $u$ is the unique solution for the problem \eqref{m1}. Recall that
\begin{equation}
\|S (f)\| \leq C \|f\|_{X^*}, \quad f \in X^*,
\end{equation}
holds true for some $C > 0$ showing that $S$ is a continuous operator. It is easy to verify that $S = L^{-1}$.
Now, let us prove that $S = L^{-1}:X^{\prime}\rightarrow X$ is continuous. Let $(f_{n})\subset X^{\prime}$ be such that $f_{n}\rightarrow f$ in $X^{\prime}$. Thus, there exist $u\in X$ and $(u_{n})\subset X$ such that $L(u_{n})=f_{n}$ and $L(u)=f$. We split the proof into three claims.

\vspace{0,3cm}

\noindent\textit{Claim 1.} $(u_{n})$ is bounded in $X$.

\vspace{0,3cm}

In fact, it follows from \eqref{m2} and the fact that $(f_{n})$ is bounded in $X^{\prime}$ that
 \[
  \ell\min\{\Vert u_{n}\Vert^{\ell},\Vert u_{n}\Vert^{m} \}\leq \langle L(u_{n}),u_{n}\rangle=\langle f_{n},u_{n}\rangle\leq C\Vert u_{n}\Vert,
 \]
which implies that $(u_{n})$ is bounded in $X$.

\vspace{0,3cm}

\noindent\textit{Claim 2.} $u_{n}(x)\rightarrow u(x)$, almost everywhere in $\mathbb{R}^{N}$.

\vspace{0,3cm}

Since $(u_{n})$ is bounded, one may deduce that
 \begin{eqnarray}
   0 & \leq & \int V_{0}[\varphi(u_{n})u_{n}-\varphi(u)u][u_{n}-u]\,\mathrm{d}x\nonumber\\
       & \leq & \int V(x) [\varphi(u_{n})u_{n}-\varphi(u)u][u_{n}-u]\,\mathrm{d}x\nonumber\\
       & \leq & \langle L(u_{n})-L(u),u_{n}-u\rangle\nonumber\\
       & = & \langle f_{n}-f,u_{n}-u\rangle\nonumber\\
       & \leq & C\Vert f_{n}-f\Vert\rightarrow0\nonumber.
 \end{eqnarray}
Thus, we conclude that
 \[
  0\leq (\varphi(u_{n})u_{n}-\varphi(u)u)(u_{n}-u)\rightarrow0, \quad \mbox{in } L^{1}(\mathbb{R}^{N}),
 \]
which implies that $u_{n}(x)\rightarrow u(x)$ almost everywhere in $\mathbb{R}^{N}$.

\vspace{0,3cm}

\noindent\textit{Claim 3.} There exist $(g_{1,n})\subset L^{1}(\mathbb{R}^{2N})$, $g_1\in L^{1}(\mathbb{R}^{2N})$, $(g_{2,n})\subset L^{1}(\mathbb{R}^{N})$ and  $g_2\in L^{1}(\mathbb{R}^{N})$ such that
 \[
  \mathcal{G}_1(u_{n}):=\Phi(\vert D_{s}u_{n}-D_{s}u\vert)\leq g_{1,n}\qquad\mbox{\and}\qquad{\mathcal{G}_2(u_{n}):=V(.)\Phi(u_n-u)\leq g_{2,n}}
 \]
satisfying the following identities
 \[
  \lim_{n \to \infty} g_{1,n}(x,y)=g_1(x,y), \quad \mbox{almost everywhere in } \mathbb{R}^{2N}
 \]
 and
 	\[
 	\lim_{n \to \infty} g_{2,n}(x)=g_2(x), \quad \mbox{almost everywhere in } \mathbb{R}^{N}.
 	\]
In fact, by using the fact that $\Phi\in\Delta_{2}$ and $\Phi$ is convex, we deduce that
 \begin{align*}
  \mathcal{G}_1(u_{n})\leq 2^{m-1}&\left[\frac{1}{\ell}\varphi(\vert D_{s}u_{n}\vert)\vert D_{s}u_{n}\vert+\Phi(\vert D_{s}u\vert) \right]:=g_{1,n}
 \end{align*}
 	\begin{align*}
 	\mathcal{G}_2(u_{n})\leq 2^{m-1}&\left[\frac{1}{\ell}V(x)\varphi(u_{n})u_{n}^{2}+V(x)\Phi(u) \right]:=g_{2,n}
 	\end{align*}
Notice also that
 \[
  \lim_{n \to \infty}g_{1,n}(x)=2^{m-1}\left[\frac{1}{\ell}\varphi(\vert D_{s}u\vert)\vert D_{s}u\vert +\Phi(\vert D_{s}u\vert) \right]:=g_1,
 \]
 	\[
 	\lim_{n \to \infty}g_{2,n}(x)=2^{m-1}\left[\frac{1}{\ell}V(x)\varphi(u)u^{2}+V(x)\Phi(u) \right]:=g_2,
 	\]
which proves \textit{Claim 3}.

Finally, since $u_{n}(x)\rightarrow u(x)$ almost everywhere in $\mathbb{R}^{N}$ and $(u_{n})$ is bounded in $X$, it follows that $u_{n}\rightharpoonup u$ weakly in $L_{\Phi}(\mathbb{R}^{N})$. Since $(D_s u_n)$ is a bounded sequence in $L_{\Phi}(\mathbb{R}^{2N}, \mathrm{d}\mu)$ it follows that $D_s u_n \rightharpoonup w$ in $L_{\Phi}(\mathbb{R}^{2N}, \mathrm{d}\mu)$ for some $w \in L_{\Phi}(\mathbb{R}^{2N}, \mathrm{d}\mu)$. In fact, $w = D_s u$ a. e. in $\mathbb{R}^{2N}$ proving that $D_{s}u_{n}\rightharpoonup D_s u$ in $L_{\Phi}(\mathbb{R}^{2N};\mathrm{d}\mu)$. Here was used the fact that the set $\{ (x,y) \in \mathbb{R}^N \times \mathbb{R}^N :  x = y\}$ has zero Lebesgue measure in $\mathbb{R}^N \times \mathbb{R}^N$ and $D_s u_n \to D_s u$ a.e in $\mathbb{R}^{2N}$.  Thus, one may conclude that $\langle f_{n},u_{n}\rangle\rightarrow\langle f,u\rangle$, as $n\rightarrow\infty$. Therefore, it follows from Lebesgue generalized Theorem that
 \[
  \iint \Phi(\vert D_{s}u_{n}-D_{s}u\vert)\,\mathrm{d}\mu +\int V(x)\Phi(\vert u_{n}-u\vert)\,\mathrm{d}x\rightarrow0, \quad \mbox{as } n\rightarrow\infty.
 \]
Since $\Phi\in\Delta_{2}$, we conclude that $u_{n}=L^{-1}(f_n)\rightarrow u=L^{-1}(f)$ in $X$. This ends the proof.

\subsection{Proof of Theorem \ref{lions} (Lions' Lemma type result)}

	Let $(u_{n})\subset X$ be satisfying \eqref{lionss}.  Since $\Psi<<\Phi_{*}$, for given $\varepsilon>0$, there exists $T>0$ such that
	\begin{equation}\label{l1}
	\frac{\Psi(\vert t\vert)}{\Phi_{*}(\vert t\vert)}\leq \frac{\varepsilon}{3\kappa}, \quad \vert t\vert \geq T,
	\end{equation}
	where
	 \[
	 \kappa=\sup_n\int_{\mathbb{R}^{N}}\Phi_{*}(\vert u_{n}\vert)\,\mathrm{d}x.
	 \]
	From \eqref{mla1b}, there exists $\delta>0$ such that
	\begin{eqnarray}\label{l3}
	\frac{\Psi(|t|)}{\Phi(|t|)}\leq \frac{\varepsilon}{3\theta},\qquad |t|<\delta,
	\end{eqnarray}
	where $$\theta:=\sup_{n}\int_{\mathbb{R}^N}\Phi(|u_n|)\,\mathrm{d}x.$$
	Let us write % Thus, we have that $u_{n}\rightarrow0$ strongly in $L_{\Phi}(B_{r}(y))$.
	\begin{equation}\label{l2}	
	 \int_{\mathbb{R}^{N}}\Psi(\vert u_{n}\vert)\,\mathrm{d}x=\int_{\{\vert u_{n}\vert\leq \delta\}}\Psi(\vert u_{n}\vert)\,\mathrm{d}x+\int_{\{\delta<|u_n|<T\}}\Psi(\vert u_{n}\vert)\,\mathrm{d}x+\int_{\{\vert u_{n}\vert\geq T\}}\Psi(\vert u_{n}\vert)\,\mathrm{d}x.
	\end{equation}
	In view of \eqref{l1} we have
	\begin{equation}\label{mlb0}
	\int_{\{\vert u_{n}\vert\geq T\}}\Psi(\vert u_{n}\vert)\,\mathrm{d}x\leq \frac{\varepsilon}{3\kappa}\int_{\mathbb{R}^{N}}\Phi_{*}(\vert u_{n}\vert)\,\mathrm{d}x\leq \frac{\varepsilon}{3}.
	\end{equation}
    It follows from \eqref{l3} that
	\begin{eqnarray}\label{mlb1}
		\int_{\{\vert u_{n}\vert\leq\delta\}}\Psi(\vert u_{n}\vert)\,\mathrm{d}x \leq \frac{\varepsilon}{3\theta}\int_{\{|u_n|\leq\delta\}}\Phi(|u_n|)\,\mathrm{d}x\leq \frac{\varepsilon}{3}.
 	\end{eqnarray}

 	There are two cases to consider. In the first one we suppose that
 	\begin{equation}\label{ed1}
 	\lim_{n \to \infty} |\{\delta<|u_n|<T\}|= 0.
 	\end{equation}
 	Thus, there exists $n_{0}\in\mathbb{N}$ such that
 	 \begin{equation}\label{ml10}
 	  |\{\delta<|u_n|<T\}|<\dfrac{\Phi(\delta)}{\Psi(T)\Phi(T)}\dfrac{\varepsilon}{3}, \quad n\geq n_{0}.
 	 \end{equation}
 	Hence, we obtain that
 	\begin{equation}\label{ed2}
    |\{\delta<|u_n|<T\}| \leq \frac{1}{\Phi(\delta)}\int_{\{\delta<|u_n|<T\}}\Phi(u_n)\,\mathrm{d}x \leq \frac{\Phi(T)}{\Phi(\delta)}  |\{\delta<|u_n|<T\}|.
 	\end{equation}
 	For $n\geq n_{0}$, it follows from \eqref{ml10} and \eqref{ed2} that
 	 \begin{equation}\label{ml12}
 	 \int_{\{\delta<|u_n|<T\}}\Psi(u_{n})\,\mathrm{d}x\leq \dfrac{\Psi(T)}{\Phi(\delta)}\int_{\{\delta<|u_n|<T\}}\Phi(u_{n})\,\mathrm{d}x<\dfrac{\varepsilon}{3}.
 	 \end{equation}
 	Therefore, by using \eqref{mlb0}, \eqref{mlb1} and \eqref{ml12}, we deduce that
 	\begin{equation}
 		\int \Psi(u_n)\,\mathrm{d}x \leq \varepsilon
 	\end{equation}
 	holds true for each $\varepsilon > 0$. This finishes the proof for the first case.

 	In the second case, up to a subsequence, we assume that
 	\begin{equation}
 	\lim_{n \to \infty} |\{\delta<|u_n|<T\}| = L \in (0, \infty).
 	\end{equation}
   Let us prove that this case does not hold. For this purpose, we prove the following claim:

 	{\bf Claim:} There exist $y_0\in\mathbb{R}^N$ and $\sigma > 0$ such that \begin{equation}
 	0 < \sigma \leq |\{\delta<|u_{n}|<T\}\cap B_r(y_0)|
 	\end{equation}
 	holds true for a subsequence of $(u_n)$ which is also labeled as $(u_n)$. The proof follows arguing by contradiction. Indeed, for each $\varepsilon > 0, k \in \mathbb{N}$ we obtain that
 	\begin{equation}\label{ed7}
 	|\{\delta<|u_{n}|<T\}\cap B_r(y)| < \frac{\varepsilon}{2^k}
 	\end{equation}
 	holds for all $y \in \mathbb{R}^N$. Notice also that the last estimate holds for any subsequence of $(u_n)$. Without loss of generality we take just the sequence $(u_n)$. Now, choose $(y_k) \subset \mathbb{R}^N$ such that $ \cup_{k=1}^{\infty} B_r (y_k) = \mathbb{R}^N$ and using \eqref{ed7}, we write
 	\begin{eqnarray}
 	|\{\delta<|u_n|<T\}| &=& |\{\delta<|u_n|<T\} \cap (\cup_{k=1}^{\infty} B_r (y_k))|  \nonumber \\
 	&\leq&
 	\sum_{k=1}^{\infty} 	|\{\delta<|u_n|<T\} \cap B_r(y_k)| \leq
 	\sum_{k=1}^{\infty} 	\frac{\varepsilon}{2^k} = \varepsilon
 	\end{eqnarray}
 	where $\varepsilon > 0$ is arbitrary. Up to a subsequence it follows from the last estimate that
 	\begin{equation}
 	0 < L  = \lim_{n \to \infty} |\{\delta<|u_n|<T\}| \leq \varepsilon
 	\end{equation}
 	which does not make sense for $\varepsilon \in (0, L)$. Thus the proof of Claim follows.
 	
 	At this stage, by using Claim and \eqref{lionss}, we observe that
 	\begin{eqnarray}
 	0 &<& \sigma \leq |\{\delta<|u_{n}|<T\}\cap B_r(y_0)| \leq \frac{1}{\Phi(\delta)} \int_{B_r(y_0)} \Phi(u_n)\,\mathrm{d}x \nonumber \\
 	&\leq& \frac{1}{\Phi(\delta)} \sup_{y \in \mathbb{R}^N} \int_{B_r(y)} \Phi(u_n)\,\mathrm{d}x  \to 0 \nonumber
 	\end{eqnarray}
 	as $n \to \infty$. This contradiction proves that second case is impossible. In other words, we prove that $L = 0$ is always verified. Hence, our result follows from the first  case. This ends the proof.

\section{Part II -- Applications}\label{s4}

\subsection{The variational framework and the Nehari manifold}\label{s1}

In this section we shall consider the variational framework for our main problem. Namely, associated to Problem \eqref{p1} we have the energy functional $J:X\rightarrow\mathbb{R}$ defined by
 \[
  J(u):=\iint \Phi\left(|D_s u|\right)\,\mathrm{d}\mu+\int V(x)\Phi(|u|)\;\mathrm{d}x-\int F(x,u)\;\mathrm{d}x,
 \]
where $\mathrm{d}\mu(x,y):=\frac{ \mathrm{d}x\mathrm{d}y}{|x-y|^N}$. It follows from \ref{f0} and \ref{f1} that for any $\varepsilon>0$, there exists $C_{\varepsilon}>0$ such that
 \begin{equation}\label{growth1}
  F(x,t)\leq \varepsilon\Phi(t)+C_{\varepsilon}\Phi_{*}(t), \quad \mbox{for all} \hspace{0,2cm} (x,t)\in\mathbb{R}^{N}\times[0,+\infty).
 \end{equation}
Thus, by using standard arguments we have that $J\in C^{1}(X)$ and its derivative is given by
 \begin{align*}
  \langle J^{\prime}(u),h\rangle&=\iint\varphi\left(|D_s u|\right)|D_s u|  D_s h \,\mathrm{d}\mu+\int V(x)\varphi(|u|)uh\;\mathrm{d}x-\int f(x,u)h\,\mathrm{d}x,\quad u, h \in X.
 \end{align*}

In this section we shall consider the Nehari method for the energy functional $J$. Recall that associated to Problem \eqref{p1} we have the Nehari manifold
 \[
  \mathcal{N}:=\left\{u\in X\backslash\{0\}:J^{\prime}(u)u=0\right\}.
 \]

\noindent Notice that for any $u\in\mathcal{N}$ we obtain
 \begin{equation}\label{emj3}
  \iint\varphi\left(|D_s u|\right)|D_s u|^{2}\,\mathrm{d}\mu+\int V(x)\varphi(|u|)u^{2}\;\mathrm{d}x=\int f(x,u)u\;\mathrm{d}x.
 \end{equation}

\begin{lemma}\label{n1}
	$\mathcal{N}$ is a $C^{1}$-manifold.
\end{lemma}
\begin{proof}
	Let us define the $C^{1}$-functional $\tilde{J}:X\backslash\{0\}\rightarrow\mathbb{R}$ given by $\tilde{J}(u)=J^{\prime}(u)u$, that is
	 \[
	  \tilde{J}(u)=\iint\varphi\left(|D_s u|\right)|D_s u|^{2}\,\mathrm{d}\mu+\int V(x)\varphi(|u|)u^{2}\;\mathrm{d}x-\int f(x,u)u\;\mathrm{d}x.
	 \]
	If $u\in\mathcal{N}$, then by using \eqref{emj2} we can deduce that
	 \begin{align*}
	  \tilde{J}^{\prime}(u)u\leq m\left\{\iint\varphi\left(|D_s u|\right)|D_s u|^{2}\,\mathrm{d}\mu+\int V(x)\varphi(|u|)u^{2}\;\mathrm{d}x\right\}\\
	  -\int[f^{\prime}(x,u)u^{2}+f(x,u)u]\;\mathrm{d}x,
	 \end{align*}
	which together with \eqref{emj3} implies that
	 \begin{equation}\label{emj4}
	  \tilde{J}^{\prime}(u)u\leq \int[(m-1)f(x,u)u-f^{\prime}(x,u)u^{2}]\;\mathrm{d}x.
	 \end{equation}
	On the other hand, if $t>0$ then it follows from assumption \ref{f4} that
	 \begin{equation*}
	  0<\frac{d}{dt}\left[\frac{f(x,t)}{t^{m-1}}\right]=\frac{t^{m-3}[f^{\prime}(x,t)t^{2}-(m-1)f(x,t)t]}{t^{2(m-1)}}.
     \end{equation*}
    Analogously we obtain the same positivity when $t<0$. Thus, it follows from \eqref{emj4} that $\tilde{J}^{\prime}(u)u<0$, for all $u\in\mathcal{N}$. Therefore, $0$ is a regular value of $\tilde{J}$ which implies that $\mathcal{N}$ is a $C^{1}$-manifold.
\end{proof}

\begin{lemma}\label{n1}
The Nehari manifold $\mathcal{N}$ is away from zero. In particular, $\mathcal{N}$  is closed set in $X$.
\end{lemma}
\begin{proof}
It follows from \eqref{growth1} and \eqref{emj3} that
\begin{equation}
\iint\varphi\left(|D_s u|\right)|D_s u|^{2}\,\mathrm{d}\mu+\int V(x)\varphi(|u|)u^{2}\;\mathrm{d}x \leq \varepsilon \int \Phi(u) \mathrm{d}x + C_\varepsilon \int \Psi (u) \mathrm{d}x, \quad u \in \mathcal{N}.
\end{equation}
  Note that, by using assumptions $(\varphi_3)$ and $(f_0)$ imply that
\begin{align*}
\iint \Phi\left(|D_s u|\right)\,\mathrm{d}\mu+\int V(x)\Phi(|u|)\;\mathrm{d}x& \leq C_{\epsilon} \int \Psi(|u|)\,\mathrm{d}x, \quad u \in \mathcal{N}.
\end{align*}
The last inequality together with Lemmas \ref{V0ineq} and \ref{V.ineq} leads to
$$
\ell\left(\xi^-([u]_{s,\Phi})+\xi^-(\|u\|_{V,\Phi})\right)\leq C\xi^+_\Psi(\|u\|),
$$
where $\xi^+_\Psi(t)=\max\{t^{\ell_\Psi},t^{m_\Psi}\}$.
As a consequence, there exists $c > 0$ such that
$\|u\| \geq c$ holds true for any $u \in \mathcal{N}$. Hence $\mathcal{N}$ is a closed set.
This ends the proof.
\end{proof}

\begin{lemma}\label{n2}
	For given $u\in E\backslash\{0\}$, there exists a unique $t_{0}>0$, depending only on $u$, such that
	 \[
	  t_{0}u\in\mathcal{N} \quad \mbox{and} \quad J(t_{0}u)=\max_{t\geq0}J(tu).
	 \]
\end{lemma}
\begin{proof}
	For given $u\in E\backslash\{0\}$, let us define $g:[0,+\infty)\rightarrow\mathbb{R}$ by $g(t)=J(tu)$. Notice that
	 \begin{equation}
	  g^{\prime}(t)t=\iint\varphi\left(|D_s tu|\right)|D_s tu|^{2}\,\mathrm{d}\mu+\int V(x)\varphi(|tu|)(tu)^{2}\;\mathrm{d}x-\int f(x,tu)tu\;\mathrm{d}x.
	 \end{equation}
	Thus, $g^{\prime}(t)=0$ if and only if $tu\in\mathcal{N}$. Let us consider $t\in(0,1)$. By using \eqref{growth1} we deduce that
	 \[
	  \frac{g(t)}{t^{m}}\geq \frac{1}{t^{m}}\left\{\iint\Phi\left(|D_s tu|\right)\,\mathrm{d}\mu+\int V(x)\Phi(|tu|)\;\mathrm{d}x \right\}-\frac{\varepsilon}{t^{m}}\int \Phi(t|u|)\;\mathrm{d}x-\frac{C_{\varepsilon}}{t^{m}}\int\Phi_{*}(t|u|)\;\mathrm{d}x,
	 \]
	which together with Remark \ref{lambda1} and Proposition~\ref{lema_naru} implies that
	 \[
	  \frac{g(t)}{t^{m}}\geq \left(1-\frac{\varepsilon}{\lambda_{1}}\right)\left\{\iint\Phi\left(|D_s u|\right)\,\mathrm{d}\mu+\int V(x)\Phi(|u|)\;\mathrm{d}x \right\}-C_{\varepsilon}t^{l_{*}-m}\max\{\|u\|_{\Phi_{*}}^{l_{*}},\|u\|_{\Phi^{*}}^{m_{*}}\}.
	 \]
	Since $l_{*}>m$ and taking $\varepsilon>0$ small such that $1-\varepsilon/\lambda_{1}>0$, we conclude that $g(t)>0$ for $t>0$ sufficiently small. On the other hand, for $t>1$ it follows from Proposition~\ref{lema_naru} that
	 \begin{equation}\label{emj5}
	  \frac{g(t)}{t^{m}}\leq \iint\Phi\left(|D_s u|\right)\,\mathrm{d}\mu+\int V(x)\Phi(|u|)\;\mathrm{d}x-\int\frac{F(x,tu)}{|tu|^{m}}|u|^{m}\;\mathrm{d}x.
	 \end{equation}
	Moreover, by using \ref{f3} and Fatou's Lemma we have that
	 \begin{equation}\label{emj6}
	  \lim_{|s|\rightarrow+\infty}\frac{F(x,s)}{|s|^{m}}=+\infty.
	 \end{equation}
	Combining \eqref{emj5} and \eqref{emj6} we conclude that $g(t)<0$ for $t>0$ sufficiently large. Therefore, $g$ admits maximum points in $(0,+\infty)$. Note that any maximum point $t>0$ satisfies
	 \begin{equation}\label{emj7}
	  \frac{1}{t^{m}}\iint\varphi\left(|D_s tu|\right)|D_s tu|^{2}\,\mathrm{d}\mu+\int V(x)\frac{\varphi(|tu|)}{t^{m}}(tu)^{2}\;\mathrm{d}x=\int \frac{f(x,tu)}{t^{m}}tu\;\mathrm{d}x.
	 \end{equation}
	In view of \ref{f4} one may conclude that the term on the righ-hand side is strictly increasing for $t>0$. In order to consider the terms on the left-hand side of \eqref{emj7}, we claim that the function $t\mapsto \varphi(|t|)/|t|^{m-3}t$ is decreasing for $t>0$. In fact, in view of hypothesis $(\varphi_4)$ we obtain
	 \[
	  \frac{d}{dt}\left[\frac{\varphi(t)}{t^{m-2}}\right]=\frac{\displaystyle\varphi(t)}{\displaystyle t^{m-1}}\left[\frac{\displaystyle\varphi^{\prime}(t)t}{\displaystyle\varphi(t)}-(m-2) \right]\leq \frac{\displaystyle\varphi(t)}{t^{m-2}}[m-2-(m-2)]=0.
	 \]
	Notice that the second term in the left-hand side of \eqref{emj7} is
	 \[
	  \int V(x)\frac{\varphi(|tu|)}{t^{m}}(tu)^{2}\;\mathrm{d}x=\int V(x)\frac{\varphi(tu)}{|tu|^{m-2}}|u|^{m}\;\mathrm{d}x,	
	 \]
	which is decreasing on $t>0$ by the above discussion. Moreover, by similar arguments the following function
	 \[
	  t\longmapsto\frac{\displaystyle\varphi\left(|D_s tu|\right)}{|D_s tu|^{m-2}}\frac{|u(x)-u(y)|^{m}}{|x-y|^{N+sm}},
	 \]
	is decreasing on $t>0$, which implies that the first term of \eqref{emj7} is decreasing. Therefore, there exists a unique $t_{0}>0$, depending only on $u$, such that $g^{\prime}(t_{0})=0$. This ends the proof.
\end{proof}

\begin{lemma}\label{coercive}
Let $(\varphi_3)$, $(f_0)$ and $(f_5)$ hold. Then there exists $\alpha >0$ such that
 $$
 J(u) \geq \alpha, \quad\text{for all}\ u\in\mathcal{N}.
 $$
 Furthermore, the functinal $J$ is coercive in the Nehari manifold $\mathcal{N}$.
\end{lemma}
\begin{proof}
It follows from Lemma \ref{n1} that there exists $c > 0$ such that
$\|u\| \geq c$ holds true for any $u \in \mathcal{N}$.
On the other hand, for any $ u\in\mathcal{N}$, the conditions $(\varphi_3)$ and \eqref{growth1} give
\begin{align}\label{est}
   J(u)&\geq \iint \Phi\left(|D_s u|\right)\,\mathrm{d}\mu+\int V(x)\Phi(|u|)\;\mathrm{d}x-\frac{1}{\theta}\int f(x,u)u\;\mathrm{d}x \nonumber \\
    &= \iint \Phi\left(|D_s u|\right)\,\mathrm{d}\mu+\int V(x)\Phi(|u|)\;\mathrm{d}x-\frac{1}{\theta}\iint\varphi\left(|D_s u|\right)|D_s u|^{2}\,\mathrm{d}\mu-\frac{1}{\theta}\int V(x)\varphi(|u|)u^{2}\;\mathrm{d}x \nonumber \\
    &\geq \left(1-\frac{m}{\theta}\right)\left[\iint \Phi\left(|D_s u|\right)\,\mathrm{d}\mu+\int V(x)\Phi(|u|)\;\mathrm{d}x\right] \nonumber \\
    &\geq \left(1-\frac{m}{\theta}\right)\left[\xi^-([u]_{s,\Phi})+\xi^-(\|u\|_{V,\Phi})\right].
\end{align}
Since $\theta>m$, we have
$
J(u)>\alpha_2,
$
for some $\alpha_2>0$. In particular, by using \eqref{est}, we obtain that
$J(u) \rightarrow \infty$ as $\|u\| \rightarrow \infty$ with $u \in \mathcal{N}$. Then, the functional $J$ is coercive over the Nehari manifold. This completes the proof.
\end{proof}

Let us introduce the ground state energy level given by
 \[
  c_{\mathcal{N}}:=\inf_{u\in\mathcal{N}}J(u).
 \]
Our goal is to ensure that level $c_{\mathcal{N}}$ is attained by a function $u \in X$, which will be a critical point for the energy function $J$. Notice that $c_{\mathcal{N}}>0$ is satisfied whose the proof can be done using the fact that the functional $J$ is coercive in the Nehari manifold, see Proposition \ref{coercive}.

\subsection{Proof of Theorem \ref{A} (Unbounded potential)}

In order to prove existence of solution for our main problem, we borrow and adapt some ideas from \cite{bw}.
Let $(u_{n})\subset \mathcal{N}$ be a $(PS)_{c_{\mathcal{N}}}$-sequence, that is,
\begin{equation}\label{psc}
J(u_{n})\rightarrow c_{\mathcal{N}} \quad \mbox{and} \quad J^{\prime}(u_{n})u_{n}\rightarrow0.
\end{equation}
In view of Lemma \ref{coercive} we have that $(u_{n})$ is bounded in $X$. Thus, up to a subsequence, $u_{n}\rightharpoonup u_{0}$ weakly in $X$ for some $u_0 \in X$. In light of Theorem \ref{compact} it follows that $u_{n}\rightarrow u_{0}$ strongly in $L_{\Psi}(\mathbb{R}^{N})$ and consequently $u_{n}(x)\rightarrow u_{0}(x)$, a.e. in $\mathbb{R}^{N}$ and there exists $h\in L_{\Psi}(\mathbb{R}^{N})$ such that $\vert u_{n}\vert\leq h$ in $\mathbb{R}^N$. We claim that
 \begin{equation*}
  f(x,u_{n})\rightarrow f(x,u_{0}), \quad \mbox{strongly in } L_{\Psi}(\mathbb{R}^{N}).
 \end{equation*}
By using $(f_{0})$ and the fact that $\Psi(t\psi(t))\leq 2\Psi(t)$, one can deduce
 \begin{equation*}
  \Psi(\vert f(x,u_{n})-f(x,u_{0})\vert)\leq C[\Psi(h)+\Psi(u_{0})]\in L_{1}(\mathbb{R}^{N}).
 \end{equation*}
The claim follows by applying the Lebesgue Dominated Convergence Theorem. Moreover, by using the fact that  $u_{n}\rightharpoonup u_{0}$ weakly in $X$ we infer also that  $J^{\prime}(u_0) h = 0$ for each $h \in X$. In particular, $u_0$ is a weak solution to the elliptic problem \eqref{p1} which implies that $S(f(\cdot, u_0)) = u_0$.

In view of Section \ref{sec5}, we have that $L:=(-\Delta_{\Phi})^{s}+V(x)\varphi(\cdot)$ is a homeomorphism. Recall also that
\begin{equation}\label{e11}
  \langle L(u_{n}), h \rangle =  \langle J^{\prime}(u_{n}), h \rangle + \int f(x,u_{n}) h \rightarrow \int f(x,u_{0}) h, \quad \mbox{as} \quad n\rightarrow\infty
 \end{equation}
 and
 \begin{equation}\label{e12}
 \langle L(u_{n}), h \rangle \rightarrow  \langle L(u_0), h \rangle, \quad \mbox{as} \quad n\rightarrow\infty
 \end{equation}
holds true for any $h \in X$.
Thus, by using the fact that $L$ is a homeomorphism and taking into account \eqref{e11} and \eqref{e12}, we obtain
$$\lim_{n \to \infty}u_{n} = \lim_{n \rightarrow \infty}  S(f(x,u_{n})) = S(f(x,u_{0})) = u_0.$$ Therefore, $u_{n}\rightarrow u_{0}$ strongly in $X$. Now, by using a density argument mention also that $J^{\prime}(u)h=0$ holds for all $h\in W^{s, \Phi}(\mathbb{R}^N)$. Furthermore, using the fact $u_{n}\rightarrow u_{0}$ strongly in $X$, we deduce that that $J(u)=c_{\mathcal{N}}$, that is, $u$ is a ground state solution.
Recall that the main tool used to obtain the assertion just above is the compact embedding $X\hookrightarrow L^{\Phi}(\mathbb{R}^{N})$. This finishes the proof.

\subsection{Proof of Theorem \ref{B} (Bounded potential)}\label{s5}

Since the potential $V$ is bounded, the natural framework will be $X = W^{s, \Phi}(\mathbb{R}^N)$. Analogously to the proof of Theorem \ref{A}, we may obtain a minimizing sequence $(u_{n})$ such that $u_{n}\rightharpoonup u$ weakly in $X$. Notice that $u$ is in fact a critical point for $J$ by using a standard argument taking into account that $C^{\infty}_c(\mathbb{R}^N)$ is dense in $X$. Assuming that $u\neq0$ we are done. If $u=0$, then we use Lemma \ref{lions} and we define the shift sequence $v_{n}(x)=u(x+y_{n}), x \in \mathbb{R}^N$. Since $x \mapsto V(x)$ and $x \mapsto F(x,u)$ are $1$-periodic we know that $J$ is invariant under translations. Notice also that
 \[
  \Vert v_{n}\Vert = \Vert u_{n} \Vert\leq C.
 \]
holds for some constants $C > 0$.
Thus, $v_{n}\rightharpoonup v$ weakly in $X$ for some $v \in X$. In view of Lemma \ref{lions} we conclude that $v\neq0$. Moreover, due to the periodicity of the energy functional, we conclude that $J(v)=c_{\mathcal{N}}$ and $J^{\prime}(v)=0$. Here we have to prove that $J^{\prime}(v)h=0$, for all $h\in W^{s,\Phi}(\mathbb{R}^{N})$. This ends the proof.

%\begin{acknowledgement}
%	Research supported in part by INCTmat/MCT/Brazil, CNPq and CAPES/Brazil. The authors would like to express their sincere gratitude to the referee for carefully reading the manuscript and valuable comments and suggestions.
%\end{acknowledgement}

\end{document}